\documentclass[dvips,11pt]{article}
\usepackage[latin1]{inputenc}
\usepackage{graphicx,psfrag}
\usepackage{amsmath,amsthm,amsfonts}
\usepackage{algorithm,algorithmic}
\usepackage{enumerate}
\usepackage[ps2pdf,colorlinks]{hyperref}

\theoremstyle{plain}
\newtheorem{theorem}{Theorem}
\newtheorem{lemma}{Lemma}
\theoremstyle{definition}
\newtheorem{example}{Example}
\newtheorem{remark}{Remark}

\newcommand{\reff}[1]{(\ref{#1})}
\newcommand{\mailto}[1]{\href{mailto:#1}{\nolinkurl{#1}}}
\newcommand{\tsum}{\textstyle\sum}

\newcommand{\Z}{\mathbb{Z}}

\newcommand{\E}{\operatorname{E}}

\newcommand{\XLB}{X^{\rm{LB}}}

\newcommand{\Sprod}{S_1 \times S_2}
\newcommand{\Rst}{R_{\rm{st}}}
\newcommand{\simst}{\sim_{\rm{st}}}

\newcommand{\lest}{\le_{\rm{st}}}

\newcommand{\RSum}{R^{\rm{sum}}}
\newcommand{\RCoord}{R^{\rm{coord}}}
\newcommand{\RMonoCoord}{R^{\rm{coord}}_1}
\newcommand{\RMonoSum}{R^{\rm{sum}}_1}
\newcommand{\RMonoMin}{R^{\rm{min}}_1}

\begin{document}

\title{Computational methods for stochastic \\ relations and Markovian couplings}

\author{
 Lasse Leskelä\thanks{
 Postal address:
 Department of Mathematics and Systems Analysis,
 Helsinki University of Technology,
 PO Box 1100, 02015 TKK, Finland.
 Tel: +358 9 451 3040.
 URL: \url{http://www.iki.fi/lsl/} \quad
 Email: \protect\mailto{lasse.leskela@iki.fi}}
}
\date{31 May 2009}
\maketitle

\maketitle

\begin{abstract}
Order-preserving couplings are elegant tools for obtaining robust estimates of the time-dependent
and stationary distributions of Markov processes that are too complex to be analyzed exactly. The
starting point of this paper is to study stochastic relations, which may be viewed as natural
generalizations of stochastic orders. This generalization is motivated by the observation that for
the stochastic ordering of two Markov processes, it suffices that the generators of the processes
preserve some, not necessarily reflexive or transitive, subrelation of the order relation. The
main contributions of the paper are an algorithmic characterization of stochastic relations
between finite spaces, and a truncation approach for comparing infinite-state Markov processes.
The methods are illustrated with applications to loss networks and parallel queues.
\end{abstract}


\vspace{1ex}


\section{Introduction}
Comparison techniques based on stochastic orders~\cite{muller2002,shaked2007,szekli1995} are key
to obtaining upper and lower bounds for complicated random variables and processes in terms of
simpler random elements. Consider for example two ergodic discrete-time Markov processes $X$ and
$Y$ with stationary distributions $\mu_X$ and $\mu_Y$, taking values in a common ordered state
space, and denote by $\lest$ the corresponding stochastic order. Then the upper bound
\begin{equation}
 \label{eq:stationaryOrder}
 \mu_X \lest \mu_Y
\end{equation}
can be established~\cite{kamae1977} without explicit knowledge of $\mu_X$ by verifying that the
corresponding transition probability kernels $P_X$ and $P_Y$ satisfy
\begin{equation}
 \label{eq:kernelOrder}
 x \le y \implies P_X(x,\cdot) \lest P_Y(y,\cdot).
\end{equation}
Analogous conditions for continuous-time Markov processes on countable spaces have been derived by
Whitt~\cite{whitt1986} and Massey~\cite{massey1987}, and later extended to more general jump
processes by Brandt and Last~\cite{brandt1994}.

Less stringent sufficient conditions for obtaining~\reff{eq:stationaryOrder} have recently been
found using a new theory stochastic relations~\cite{leskelaNow}. Two random variables are
stochastically related, denoted by $X \simst Y$, if there exists a coupling $(\hat X, \hat Y)$ of
$X$ and $Y$ such that $\hat X \sim \hat Y$ almost surely, where $\sim$ denotes some relation
between the state spaces of $X$ and $Y$. The main motivation for this definition is
that~\reff{eq:kernelOrder} is by no means necessary for~\reff{eq:stationaryOrder}; a less
stringent sufficient condition is that
\begin{equation}
 \label{eq:kernelRelation}
 x \sim y \implies P_X(x,\cdot) \simst P_Y(y,\cdot)
\end{equation}
for some, not necessarily symmetric or transitive, nontrivial subrelation of the underlying order
relation. Another advantage of the generalized definition is that $X$ and $Y$ are no longer
required to take values in the same state space, leading to greater flexibility in the search for
bounding random elements $Y$. For example, to study whether $f(X) \lest g(Y)$ for some given real
functions $f$ and $g$ defined on the state spaces of $X$ and $Y$, we may define a relation $x \sim
y$ by the condition $f(x) \le g(y)$ \cite{doisy2000}.

The rest of the paper is outlined as follows. After recalling the basic definitions,
Section~\ref{sec:StochasticRelations} presents a numerical algorithm for verifying stochastic
relations between finite spaces, together with an analysis of computational complexity.
Section~\ref{sec:MarkovProcesses} recalls how a recursive subrelation algorithm may be used to
find Markovian couplings preserving a given relation. In Section~\ref{sec:Truncation}, a new
truncation approach is presented that allows to precisely compute truncated outcomes of the
subrelation algorithm for infinite-state Markov processes with locally bounded jumps.
Section~\ref{sec:Applications} discusses applications to loss networks and parallel queues, and
Section~\ref{sec:Conclusions} concludes the paper.

\section{Stochastic relations}
\label{sec:StochasticRelations}

\subsection{General definitions}

We shall here recall the definitions of stochastic relations between countable spaces. Probability
measures $\mu$ on a countable state space $S$ shall be viewed as a probability vectors via
identifying $\mu(x) = \mu(\{x\})$. For a treatment on more general spaces, see~\cite{leskelaNow}.

A \emph{relation} between $S_1$ and $S_2$ is subset of $\Sprod$. Given a nontrivial ($R \neq
\emptyset$) relation $R$ between $S_1$ and $S_2$, we write
\[
x \sim y,
\]
if $(x,y) \in R$. The relation $R$ may equivalently be viewed as a matrix so that $R(x,y) = 1$ if
$x \sim y$ and $R(x,y) = 0$ otherwise. A \emph{coupling} of probability vectors $\mu$ on $S_1$ and
$\nu$ on $S_2$ is a probability vector $\lambda$ on $S_1 \times S_2$ with marginals $\mu$ and
$\nu$, that is,
\begin{align*}
 \sum_{y \in S_2} \lambda(x,y) &= \mu(x) \quad \text{for all $x \in S_1$},\\
 \sum_{x \in S_1} \lambda(x,y) &= \nu(y) \quad \text{for all $y \in S_2$}.
\end{align*}
For probability vectors $\mu$ on $S_1$ and $\nu$ on $S_2$ we denote
\[
\mu \simst \nu,
\]
and say that $\mu$ is stochastically related to $\nu$, if there exists a coupling $\lambda$ of
$\mu$ and $\nu$ such that
\[
 \sum_{(x,y) \in R} \mu(x,y) = 1.
\]
The relation $\Rst = \{ (\mu,\nu) : \mu \simst \nu \}$ is called the \emph{stochastic relation}
generated by $R$. Observe that two Dirac masses satisfy $\delta_{x} \simst \delta_{y}$ if and only
if $x \sim y$. In this way the stochastic relation $\Rst$ may be regarded as a natural
randomization of the underlying relation $R$.

The following result in~\cite{leskelaNow}, which is rephrased here for ease of reference, provides
an analytical method to check whether a pair of probability measures are stochastically related.

\begin{theorem}\cite{leskelaNow}
\label{the:StochasticRelation}
Two probability vectors $\mu$ and $\nu$ are stochastically related with respect to $R$ if and only
if
\begin{equation}
 \label{eq:StochasticRelation1}
 \sum_{x \in U} \mu(x) \le \sum_{y \in S_2} \left( \max_{x \in U} R(x,y) \right) \nu(y)
\end{equation}
for all finite $U \subset S_1$, or equivalently, if and only if
\begin{equation}
 \label{eq:StochasticRelation2}
 \sum_{y \in V} \nu(y) \le \sum_{x \in S_1} \left( \max_{y \in V} R(x,y) \right) \mu(x)
\end{equation}
for all finite $V \subset S_2$.
\end{theorem}

A random variable $X$ is \emph{stochastically related} to a random variable $Y$, denoted by $X
\simst Y$, if the distribution of $X$ is stochastically related to the distribution of $Y$.
Observe that $X$ and $Y$ do not need to be defined on the same probability space. Recall that a
coupling of random variables $X$ and $Y$ is a bivariate random variable whose distribution couples
the distributions of $X$ and $Y$. Hence $X \simst Y$ if and only if there exists a coupling $(\hat
X, \hat Y)$ of $X$ and $Y$ such that $\hat X \sim \hat Y$ almost surely.

\begin{example}
If $\le$ is an order (reflexive and transitive) relation on a space $S$, then the corresponding
stochastic relation $\lest$ is called a stochastic order. Using Strassen's classical
theorem~\cite{muller2002}, we see that $X \lest Y$ if and only if $\E f(X) \le \E f(Y)$ for all
positive increasing functions $f$ on $S$.
\end{example}

\subsection{Stochastic relations between finite spaces}

Let $R$ be a relation between finite spaces $S_1$ and $S_2$, and denote by $\Rst$ the
corresponding stochastic relation. Then Theorem~\ref{the:StochasticRelation} may be used to
determine whether $\mu \simst \nu$. However, this requires to check the
inequality~\reff{eq:StochasticRelation1} for all subsets of $S_1$, which is computationally
infeasible unless the spaces are small. The following result shows that less checks may be
sufficient. We shall denote the support of a probability vector $\mu$ by $U_\mu = \{x: \mu(x) >
0\}$. Moreover, we denote by $F(U,\Z_2)$ the set of vectors with components in $\{0,1\}$ indexed
by elements of $U$, which may also be identified as the set of all subsets of $U$.

\begin{theorem}
\label{the:StochasticRelationFinite}
Two probability vectors $\mu$ and $\nu$ with supports $U_\mu$ and $U_\nu$ are stochastically
related with respect to $R$ if and only if
\begin{equation}
 \label{eq:StochasticRelationFinite1}
 \sum_{x \in U_\mu} f(x) \, \mu(x) \le \sum_{y \in U_\nu} \max_{x \in U_\mu} \left[ f(x) R(x,y) \right] \nu(y)
\end{equation}
for all $f \in F(U_\mu,\Z_2)$, or equivalently, if and only if
\begin{equation}
 \label{eq:StochasticRelationFinite2}
 \sum_{y \in U_\nu} g(y) \, \nu(y) \le \sum_{x \in U_\mu} \max_{y \in U_\nu} \left[ R(x,y) g(y) \right] \mu(x)
\end{equation}
for all $g \in F(U_\nu,\Z_2)$.
\end{theorem}
\begin{proof}
In light of Theorem~\ref{the:StochasticRelation}, it suffices to show the equivalence
of~\reff{eq:StochasticRelation1} and~\reff{eq:StochasticRelationFinite1}, and the equivalence
of~\reff{eq:StochasticRelation2} and~\reff{eq:StochasticRelationFinite2}. Observe
that~\reff{eq:StochasticRelation1} directly implies~\reff{eq:StochasticRelationFinite1}, because
the members of $F(U_\mu,\Z_2)$ may be identified with the indicator functions of subsets of
$U_\mu$. To prove the converse, assume that~\reff{eq:StochasticRelationFinite1} holds, and let $U$
be an arbitrary subset of $S_1$. Define $f(x) = 1(x \in U \cap U_\mu)$. Then
\begin{align*}
 \sum_{x \in U} \mu(x)
 &= \sum_{x \in U_\mu} f(x) \mu(x) \\
 &\le \sum_{y \in U_\nu} \max_{x \in U_\mu} \left[ f(x) R(x,y) \right] \nu(y) \\
 &= \sum_{y \in S_2}   \left( \max_{x \in U_\mu \cap U} R(x,y) \right) \nu(y) \\
 &\le \sum_{y \in S_2} \left( \max_{x \in U} R(x,y) \right) \nu(y).
\end{align*}
Hence~\reff{eq:StochasticRelation1} holds. Proving the equivalence
of~\reff{eq:StochasticRelation2} and~\reff{eq:StochasticRelationFinite2} is completely analogous.
\end{proof}

\begin{algorithm}
\caption{Determining whether $\mu \simst \nu$.}
\begin{algorithmic}
 \STATE $U_\mu \gets \{x \in S_1: \mu(x)>0\}$
 \STATE $U_\nu \gets \{y \in S_2: \nu(y)>0\}$

 \IF {$\# U_\mu > \# U_\nu$}
    \STATE flip $\mu \leftrightarrow \nu$, $U_\mu \leftrightarrow U_\nu$, $S_1 \leftrightarrow S_2$
 \ENDIF

 \STATE $b \gets$ true
 \FOR {$i = 1,\dots,2^{\# U_\mu}$}
   \STATE $f \gets$ $i$-th vector in $F(U_\mu,\Z_2)$
   \STATE $v_l \gets \sum_{x \in U_\mu} f(x) \, \mu(x)$
   \STATE $v_r \gets \sum_{y \in U_\nu} \left[ \max_{x \in U_\mu} f(x) R(x,y) \right] \nu(y)$
       \IF {$v_l > v_r$}
           \STATE $b \gets$ false
           \STATE break
        \ENDIF
  \ENDFOR
  \RETURN $b$
\end{algorithmic}
\label{alg:StochasticRelation}
\end{algorithm}

Algorithm~\ref{alg:StochasticRelation} describes how Theorem~\ref{the:StochasticRelationFinite}
can be applied to numerically determine whether $\mu \simst \nu$. The interchange of the variables
in the beginning corresponds to using~\reff{eq:StochasticRelationFinite1} if the support of $\mu$
is smaller than $\nu$, and~\reff{eq:StochasticRelationFinite2} otherwise. Inspection of
Algorithm~\ref{alg:StochasticRelation} shows that the computational complexity of determining
whether $\mu \simst \nu$ is of the order
\[
 O( \max(n_1',n_2') 2^{\min(n_1',n_2')}),
\]
where $n_1'$ and $n_2'$ denote the cardinalities of the supports of $\mu$ and $\nu$. The algorithm
is very slow when both state spaces are large and $\mu$ and $\nu$ have positive mass in all
states. However, in many applications related to structured Markov chains we may assume that the
$\mu$ and $\nu$ have small supports (see Section~\ref{sec:Applications}).

\begin{remark}
The verification of $\mu \simst \nu$ can be carried out faster, if the underlying relation $R$ has
some structure that can be employed. For example, for the natural order on $S=\{1,\dots,n\}$, $\mu
\lest \nu$ can be verified in $O(n^2)$ time by checking whether $\mu K \le \nu K$ holds
coordinatewise, where $K$ is the $n$-by-$n$ lower triangular matrix such that $K(i,j) = 1(i \ge
j)$ \cite{benmamoun2007}.
\end{remark}

\section{Markov processes}
\label{sec:MarkovProcesses}

\subsection{Markovian couplings}

All state spaces in the following shall be assumed finite or countably infinite. The keep the
presentation short, all results shall be formulated for continuous-time Markov processes, which
without further mention shall be assumed nonexplosive. For a Markov process $X$ with values in $S$
we denote by $X(x,t)$ the value of the process at time $t$ given that was started at state $x$. A
Markov process $\hat X = (\hat X_1,\hat X_2)$ taking values in $\Sprod$ is called a
\emph{Markovian coupling} of $X_1$ and $X_2$ if $\hat X(x,t)$ couples $X_1(x_1,t)$ and
$X_2(x_2,t)$ for all $t$ and all $x = (x_1,x_2)$. A common approach for showing that the
time-dependent distributions of two Markov processes $X_1$ and $X_2$ are stochastically ordered,
is to find a Markovian coupling $\hat X = (\hat X_1,\hat X_2)$ of $X_1$ and $X_2$ such that
\begin{equation}
 \label{eq:OrderPreservation}
 x_1 \le x_2 \implies \hat X_1(x_1,t) \le \hat X_2(x_2,t)
\end{equation}
almost surely for all $t$~\cite{muller2002}. Observe that if~\reff{eq:OrderPreservation} holds,
then for all $t$ and all increasing functions $f$,
\begin{align*}
 \E f(X_1(x_1,t))
 &=   \E f(\hat X_1(x_1,t)) \\
 &\le \E f(\hat X_2(x_2,t))
  = \E f(X_2(x_2,t)),
\end{align*}
whenever $x_1 \le x_2$. Hence $X_1(x_1,t) \lest X_2(x_2,t)$. If both $X_1$ and $X_2$ have unique
stationary distributions, it follows by taking limits that the stationary distributions are
stochastically ordered.

Observe that~\reff{eq:OrderPreservation} is not necessary for the stochastic ordering of the
stationary distributions of $X_1$ and $X_2$. To formulate a less stringent sufficient condition,
we shall use the following definitions. Let $R$ be a relation between $S_1$ and $S_2$. A pair of
Markov processes $X_1$ in $S_1$ and $X_2$ in $S_2$ is said to \emph{stochastically preserve the
relation $R$}, if
\[
 x_1 \sim x_2 \implies X_1(x_1,t) \simst X_2(x_2,t) \quad \text{for all $t \ge 0$}.
\]
Moreover, a set $B$ is called \emph{invariant} for a Markov process $X$ if $x \in B$ implies
$X(x,t) \in B$ for all $t$ almost surely. The following results was proved in~\cite{leskelaNow},
which we rephrase here for convenience.
\begin{theorem}\cite{leskelaNow}
The following are equivalent:
\begin{enumerate}[(i)]
\item $X_1$ and $X_2$ stochastically preserve the relation $R$.
\item There exists a Markovian coupling of $X_1$ and $X_2$ for which $R$ is invariant.
\end{enumerate}
\end{theorem}

Assume now that $X_1$ and $X_2$ are Markov processes with values in an ordered space $S$, and
assume $R$ is a subrelation of the order that is stochastically preserved by $X_1$ and $X_2$. Then
$X_1(x_1,t) \lest X_2(x_2,t)$ for all $x_1 \sim x_2$. Especially, if $X_1$ and $X_2$ have unique
stationary distributions, then a sufficient condition for the stochastic ordering of the
stationary distributions is that $X_1$ and $X_2$ stochastically preserve some nontrivial
subrelation of the order relation.

\subsection{Subrelation algorithm}

Recall that a matrix $Q$ with entries $Q(x,y)$, $x,y \in S$ is called a \emph{rate matrix}, if
$Q(x,y) \ge 0$ for all $x\neq y$ and $Q(x,x) = -\sum_{y\neq x} Q(x,y)$ for all $x$. If $X$ is a
Markov process with rate matrix $Q$, then $Q(x,y)$ is the transition rate of $X$ from state $x$
into $y$, and we denote by $q(x) = -Q(x,x)$ the total transition rate of $X$ out of state $x$.

Given Markov processes $X_1$ and $X_2$ with rate matrices $Q_1$ and $Q_2$, define the
relation-to-relation mapping $M_{Q_1,Q_2}$ by
\begin{equation}
 \label{eq:M}
 M_{Q_1,Q_2}(R) = \{(x,y) \in R: (\mu_{x,y}, \nu_{x,y}) \in \Rst \},
\end{equation}
where the probability measures $\mu_{x,y}$ and $\nu_{x,y}$ are defined by
\begin{align}
 \label{eq:mu}
 \mu_{x,y}(u) &= q_{x,y}^{-1} Q_1(x,u) + \delta_{x}(u), \\
 \label{eq:nu}
 \nu_{x,y}(v) &= q_{x,y}^{-1} Q_2(y,v) + \delta_{y}(v),
\end{align}
and where $q_{x,y} = 1+q_1(x) + q_2(y)$. When there is no risk of confusion, we denote
$M_{Q_1,Q_2} = M$. Moreover, define recursively the sequence $M^k(R)$ by setting $M^0(R) = R$,
$M^{k}(R) = M(M^{k-1}(R))$ for $k \ge 1$, and denote the limit of the sequence by
\[
 M^*(R) = \cap_{k=0}^\infty M^k (R).
\]

\begin{theorem}\cite{leskelaNow}
\label{the:maximal}
The relation $M^*(R)$ is the maximal subrelation of $R$ that is stochastically preserved by $X_1$
and $X_2$. Especially:
\begin{enumerate}[(i)]
\item $X_1$ and $X_2$ stochastically preserve $R$ if and only if
    $M(R) = R$.
\item $X_1$ and $X_2$ stochastically preserve a nontrivial
    subrelation of $R$ if and only if $M^*(R) \neq \emptyset$.
\end{enumerate}
\end{theorem}

\begin{algorithm}
\caption{Computation of $R'=M(R)$.}
\begin{algorithmic}
  \STATE $R' \gets $ $n_1$-by-$n_2$ zero matrix
  \FOR {$(x,y) \in R$}
     \STATE $q \gets 1 + q^{(1)}(x) + q^{(2)}(y)$
     \STATE $\mu \gets q^{-1} Q^{(1)}(x, \cdot) + \delta_x(\cdot)$
     \STATE $\nu \gets q^{-1} Q^{(2)}(y, \cdot) + \delta_y(\cdot)$
     \STATE Check whether $\mu \simst \nu$ (use Algorithm~\ref{alg:StochasticRelation})
     \IF {$\mu \simst \nu$}
         \STATE $R'(x,y) \gets 1$
     \ENDIF
 \ENDFOR
\end{algorithmic}
\label{alg:SubrelationNew}
\end{algorithm}

When the state spaces $S_1$ and $S_2$ are finite, Algorithm~\ref{alg:SubrelationNew} describes how
to numerically compute $M(R)$, and Algorithm~\ref{alg:SubrelationLimit} describes the computation
of $M^*(R)$. Observe that for finite state spaces, Algorithm~\ref{alg:SubrelationLimit} computes
the apparently infinite intersection $\cap_{k=0}^\infty M^k(R)$ in finite time, because as long as
$M^k(R)$ and $M^{k-1}(R)$ are not equal, they differ by at least one element, and the sequence
$M^k(R)$ is decreasing.

\begin{algorithm}
\caption{Computation of $R^*= \cap_{n=0}^\infty M^k(R)$.}
\begin{algorithmic}
  \STATE $R' \gets M(R)$ (use Algorithm~\ref{alg:SubrelationNew})
  \WHILE {$R' \neq R$}
     \STATE $R \gets R'$
     \STATE $R' \gets M(R)$ (use Algorithm~\ref{alg:SubrelationNew})
  \ENDWHILE
  \STATE $R^* \gets R'$
\end{algorithmic}
\label{alg:SubrelationLimit}
\end{algorithm}

\section{Truncation approach}
\label{sec:Truncation}

\subsection{Truncation of Markov processes}

If $Q$ is a rate matrix of a Markov process on a countably infinite space $S$,
and $S_n$ is a finite subset of $S$, we define the \emph{truncation} of $Q$
into $S_n$ by
\begin{equation}
 \label{eq:RateMatrixTruncation}
 Q_n(x,y) = \left\{
 \begin{aligned}
 Q(x,y),                              &\quad x \neq y, \quad x, y \in S_n, \\
 - \sum_{y \in S_n, y \neq x} Q(x,y), &\quad x=y, \quad x \in S_n.
 \end{aligned}
 \right.
\end{equation}
We shall later approximate $Q$ by $Q_n$ and use the finite subrelation algorithm applied to $Q_n$.
To understand the approximation error, we need to study how the untruncated process may escape the
set $S_n$. Given a rate matrix $Q$ on a countable space $S$, we say that an increasing sequence of
finite sets $S_n \subset S$ is a \emph{truncation sequence} for $Q$, if $\cup_{n=0}^\infty S_n =
S$, and
\begin{equation}
 \label{eq:truncation}
 \{y: Q(x,y) > 0\} \subset S_{n+1}
\end{equation}
for all $x \in S_n$.

\begin{example}
Let $Q$ be the rate matrix of a Markov process on $\Z_+$ that is skip-free to
the right, so that $Q(i,j) = 0$ for all $j > i+1$. Then $S_n = \Z_+ \cap [0,n]$
is truncation sequence for $Q$.
\end{example}

The next result shows that truncation sequences can be constructed for most Markov processes
encountered in applications. We say that a Markov process $X$ with rate matrix $Q$ has
\emph{locally bounded jumps}, if the set $\{y: Q(x,y) > 0\}$ is finite for all $x$.

\begin{lemma}
Any rate matrix $Q$ of a Markov process with locally bounded jumps possesses a truncating
sequence.
\end{lemma}
\begin{proof}
Because $S$ is countable, we may choose an increasing sequence of finite sets $K_n$ such that
$\cup_{n=0}^\infty K_n = S$. Using this sequence we may recursively define the sets $S_n$ by
setting $S_0 = K_0$, and
\[
 S_{n+1} = S_n \cup K_n \cup J(S_n), \quad n \ge 0,
\]
where
\[
 J(S_n) = \{y: Q(x,y) > 0 \ \text{for some} \ x \in S_n\}
\]
denotes the set of states that are reachable from $S_n$ by one jump. Then $\cup_{n=0}^\infty S_n =
S$, because $K_n \subset S_{n+1}$ for all $n \ge 0$. Moreover, $J(S_n) \subset S_{n+1}$ implies
that~\reff{eq:truncation} holds for all $n$, and induction shows that the sets $S_n$ are finite,
because $Q$ has locally bounded jumps.
\end{proof}

\subsection{Truncation of stochastic relations}

Let $R$ be a relation between countably infinite state spaces $S_1$ and $S_2$.
If $S_i'$ is a finite subset of $S_i$, $i=1,2$, we define the truncation of $R$
by
\begin{equation}
 \label{eq:RelationTruncation}
 R' = R \cap (S_1' \times S_2')
\end{equation}
The corresponding stochastic relation $\Rst'$, a relation between probability
measures on the finite spaces $S_1'$ and $S_2'$, is called the \emph{truncation
of $\Rst$} into $S_1' \times S_2'$.

If $\mu_i$ is probability measure on $S_i$ having all its mass on $S_i'$, we may regard $\mu_i$ as
a probability measure $\mu_i'$ on $S_i'$ by identifying subsets of $S_i'$ as subsets of $S_i$.

\begin{lemma}
\label{the:StochasticRelationTruncated}
Let $\mu_i$ be a probability measure on $S_i$ such that $\mu_i(S_i') =
1$, $i=1,2$. Then $(\mu_1,\mu_2) \in \Rst$ if and only if $(\mu'_1,\mu'_2) \in
\Rst'$.
\end{lemma}
\begin{proof}
The claim follows using~\cite[Lemma 5.2]{leskelaNow}, after observing that $R'$
is the relation induced from $R$ by the pair $(\phi_1,\phi_2)$, where $\phi_i$
is the natural embedding of $S'_i$ into $S_i$, and $\mu_i' = \mu_i \circ
\phi_i$.
\end{proof}

\subsection{Truncated subrelation algorithm}

Given a pair of Markov processes $X_1$ and $X_2$ with rate matrices $Q_1$ and $Q_2$, and a
relation $R$ between $S_1$ and $S_2$, Algorithm~\ref{alg:SubrelationLimit} describes how to
recursively calculate $R^* = M_{Q_1,Q_2}^*(R)$ as the limit of the sequence $R^k =
M_{Q_1,Q_2}^k(R)$. If the state spaces are infinite, $R^k$ cannot be computed using finite time
and memory. Nevertheless, when the processes have locally bounded jumps, the truncations
$T_N(R^k)$ of $R^k$ into suitable truncation sets $S_{1,N} \times S_{2,N}$ can be computed
precisely, as shall be shown next.

For any relation $R$ between $S_{1,n}$ and $S_{2,n}$, let $M_n(R) =
M_{Q_{1,n},Q_{2,n}}(R)$ be the relation given by
Algorithm~\ref{alg:SubrelationNew} applied to $R$ using the truncations
$Q_{1,n}$ and $Q_{2,n}$ of $Q_1$ and $Q_2$ as defined in~\reff{eq:truncation}.
Moreover, denote by $M_\infty = M_{Q_1,Q_2}$ the corresponding untruncated
mapping.

\begin{lemma}
\label{the:OneStepTruncated}
For any pair of Markov processes with locally bounded jumps, the truncation of $M_\infty(R)$ to
$S_{1,n} \times S_{2,n}$ satisfies
\[
 T_n(M_\infty(R)) = T_n(M_{n+1}(T_{n+1}(R))) \quad \text{for all $n$}.
\]
\end{lemma}
\begin{proof}
Observe that $T_n(M_\infty(R))$ equals the set of points $(x,y) \in T_n(R)$ such that the measures
$\mu_{x,y}$ and $\nu_{x,y}$ defined in~\reff{eq:mu} and~\reff{eq:nu} are stochastically related
with respect to $R$. Because $S_{1,n}$ and $S_{2,n}$ are truncating sequences for $Q_1$ and $Q_2$,
we know that $\mu_{x,y}$ and $\nu_{x,y}$ have supports in $S_{1,n+1}$ and $S_{2,n+1}$,
respectively. Hence by Lemma~\ref{the:StochasticRelationTruncated}, $T_n(M_\infty(R))$ equals the
set of points $(x,y) \in T_n(R)$ such that $\mu_{x,y}$ and $\nu_{x,y}$ are stochastically related
with respect to $T_{n+1}(R)$. Because for $(x,y) \in T_n(R)$, the measures $\mu_{x,y}$ and
$\nu_{x,y}$ remain the same, if we replace $Q_1$ and $Q_2$ by $Q_{1,n+1}$ and $Q_{2,n+1}$
in~\reff{eq:mu} and~\reff{eq:nu}, the claim follows.
\end{proof}

\begin{theorem}
\label{the:SubrelationTruncated}
For any pair of Markov processes with locally bounded jumps, the truncation of $R^k =
M^k_\infty(R)$ to $S_{1,n} \times S_{2,n}$ satisfies
\[
 T_n(R^k) = T_n(M_{n+1}T_{n+1})\cdots (M_{n+k}T_{n+k})(R).
\]
\end{theorem}
\begin{proof}
Apply Lemma~\ref{the:OneStepTruncated} and induction.
\end{proof}

Algorithm~\ref{alg:SubrelationTruncated} describes how Theorem~\ref{the:SubrelationTruncated} can
be used to compute $T_n(R^k)$ in finite time and memory.

\begin{algorithm}
\caption{Computation of $R' = T_N(M^K(R))$.}
\begin{algorithmic}
  \STATE $R' \gets T_{N+K}(R)$
  \FOR {$k=1,\dots,K$}
     \STATE $n \gets N+K+1-k$
     \STATE $Q_{1,n} \gets$ truncation of $Q_1$ into $S_{1,n}$
     \STATE $Q_{2,n} \gets$ truncation of $Q_2$ into $S_{2,n}$
     \STATE $R' \gets T_n(R')$
     \STATE $R' \gets $ Algorithm~\ref{alg:SubrelationNew} applied to
     $(Q_{1,n},Q_{2,n},R')$
  \ENDFOR
  \STATE $R' \gets T_N(R')$
\end{algorithmic}
\label{alg:SubrelationTruncated}
\end{algorithm}

The following result is a necessary condition for finding a subrelation of $R$ that is
stochastically preserved by a pair of Markov processes.

\begin{theorem}
\label{the:NecessaryCondition}
Let $X_1$ and $X_2$ be Markov processes with locally bounded jumps. If $T_n(M^k(R)) =
\emptyset$ for some $k$, then $T_n(M^*(R)) = \emptyset$. Especially, if for all $n$, $T_n(M^k(R))
= \emptyset$ for some $k$, then $X_1$ and $X_2$ do not stochastically preserve any nontrivial
subrelation of $R$.
\end{theorem}
\begin{proof}
Because the sequence $M^k(R)$ is decreasing, and because the truncation map is monotone with
respect to set inclusion, the first claim follows. For the second claim, observe that if $T_n(R^*)
= \emptyset$ for all $n$, then because $\cup_n (S_{1,n} \times S_{2,n}) = S$, $R^* = \emptyset$.
\end{proof}

\section{Applications}
\label{sec:Applications}

\subsection{Multilayer loss networks}

\subsubsection{Overflow routing}

Consider a loss system with $K$ customer classes and two layers of servers, where layer~1 contains
$m_k$ servers dedicated to class $k$, and layer~2 consists of $n$ servers capable of serving all
customer classes. Arriving class-$k$ customers are routed to vacant servers in one of the layers,
with preference given to layer~1; or rejected otherwise. For analytical tractability, we assume
that the interarrival times and the service requirements of class-$k$ customers are exponentially
distributed with parameters $\lambda_k$ and $\mu_k$, respectively, and that all these random
variables across all customer classes are independent. For brevity, we denote $m=(m_1,\dots,m_K)$,
$\lambda = (\lambda_1,\dots,\lambda_K)$, and $\mu = (\mu_1,\dots,\mu_K)$.

\begin{figure}[h]
 \psfrag{L1}{ $\lambda_1$}
 \psfrag{L2}{ $\lambda_2$}
 \psfrag{L3}{ $\lambda_3$}
 \psfrag{M1}{$m_1$}
 \psfrag{M2}{$m_2$}
 \psfrag{M3}{$m_3$}
 \psfrag{N}{$n$}
 \psfrag{La1}{ Layer $1$}
 \psfrag{La2}{ Layer $2$}
 \centering
 \includegraphics[width=.6\textwidth]{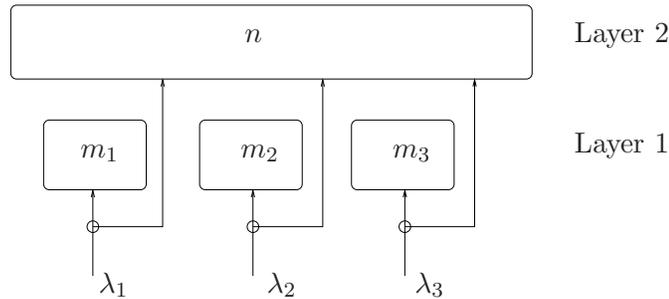}
 \caption{\label{fig:system} Two-layer loss network with three customer classes ($K=3$).}
\end{figure}

Denote by $X_{i,k}(t)$ the number of class-$k$ customers being served at layer $i$ at time $t$.
The system is described by the Markov process $X = (X_{i,k})$ taking values in
\begin{equation}
 \label{eq:stateSpace}
 S_1 = \left\{x \in \Z_+^K \times \Z_+^K: \ x_{1,k} \le m_k \ \forall k, \ \tsum_{k=1}^K x_{2,k} \le n \right\},
\end{equation}
and having the transitions
\[
 x \mapsto \left\{
 \begin{aligned}
  x + e_{1,k}, &\quad \text{at rate} \ \lambda_k 1(x_{1,k} < m_k), \\
  x + e_{2,k}, &\quad \text{at rate} \ \lambda_k 1(x_{1,k} = m_k, \tsum_{k=1}^K x_{2,k} < n), \\
  x - e_{1,k}, &\quad \text{at rate} \ \mu_k x_{1,k}, \\
  x - e_{2,k}, &\quad \text{at rate} \ \mu_k x_{2,k},
 \end{aligned}
 \right.
\]
where $e_{i,k}$ denotes the unit vector in $\Z_+^K \times \Z_+^K$ corresponding to the coordinate
direction $(i,k)$.

\subsubsection{Maximum packing}
\label{sec:modelMaximumPacking}

To approximate the original two-layer loss system, we consider a modification of the system, where
customers are redirected from layer~2 to layer~1 as soon as servers become vacant. This
corresponds to the so-called maximum packing policy introduced by Everitt and
Macfadyen~\cite{everitt1983}. Denote by $Y_k(t)$ the total number of customers in the system with
maximum packing. Then $t \mapsto Y(t)=(Y_1(t),\dots,Y_K(t))$ is a Markov process
(see~\cite{jonckheere2008}) with values in
\[
 S_2 = \left\{ y \in \Z_+^K: \tsum_{k=1}^K (y_k-m_k)_+ \le n \right\},
\]
and having the transitions
\[
 y \mapsto \left\{
 \begin{aligned}
  y + e_k, &\quad \text{at rate} \ \lambda_k 1(y + e_k \in S_2), \\
  y - e_k, &\quad \text{at rate} \ \mu_k y_k.
 \end{aligned}
 \right.
\]
The structure of the above transition rates implies that the stationary distribution $\pi_Y$ of
$Y$ is a product of Poisson distributions truncated to $S_2$~\cite{kelly1991}, so that
\[
 \pi_Y(y) = c \prod_{k=1}^K \frac{(\lambda_k/\mu_k)^{y_k}}{y_k!}, \quad y \in S_2,
\]
where the constant $c$ can be solved from $\sum_{y \in S_2} \pi_Y(y) = 1$. This product form
structure allows for fast computation of stationary performance characteristics of the maximum
packing system.

\subsubsection{Stochastic comparison}

Table~\ref{tab:Relations} illustrates the outcomes of the subrelation algorithm (computed
using~\cite{leskelaSW}) applied to various initial relations, where
\begin{align*}
 \RSum       &= \left\{ (x,y) \in S_1 \times S_2: \tsum_k (x_{1,k} + x_{2,k}) \le \tsum_k y_k \right\},\\
 \RMonoSum   &= \left\{ (x,y) \in S_1 \times S_2: \tsum_k x_{1,k} \le \tsum_k (y_k \wedge m_k) \right\}, \\
 \RMonoCoord &= \left\{ (x,y) \in S_1 \times S_2: x_{1,k} \le y_k \wedge m_k \ \text{for all} \ k
 \right\}.
\end{align*}
Note that $\RMonoSum$ relates the state pairs $(x,y)$ where the total number of layer-1 customers
corresponding to $x$ is less than that corresponding to $y$. Moreover, $\RMonoCoord$ relates the
state pairs where the populations in layer~1 are coordinatewise ordered. The entry "several" in
Table~\ref{tab:Relations} refers to running Algorithm~\ref{alg:SubrelationNew} separately for
several pseudorandom parameter combinations, and taking the intersection of the produced relations
as a final result.

\begin{table}[h]
 \centering
 \begin{tabular}[h]{c|c|c|c}
  $R$   & $(\lambda_1,\lambda_2)$ & $(\mu_1,\mu_2)$ & $M^*(R)$                                 \\
  \hline \hline
  $\RSum$ & $(1,1)$                 & $(1,1)$         & $\RSum \cap \RMonoSum \cap \RMonoMin$ \\
  $\RSum$ & several                 & $(1,1)$         & $\RSum \cap \RMonoCoord$ \\
  $\RSum$ & several                 & several         & $\emptyset$ \\ \hline
  $\RMonoSum$   & $(1,1)$           & $(1,1)$         & $\RMonoSum$ \\
  $\RMonoSum$   & several           & $(1,1)$         & $\RMonoCoord$ \\
  $\RMonoSum$   & several           & several         & $\RMonoCoord$ \\ \hline
  $\RMonoCoord$ & several           & several         & $\RMonoCoord$ \\
 \end{tabular}
 \caption{\label{tab:Relations} Outcomes of the subrelation algorithm for a network with $m_1=1, m_2=1,n=2$.}
\end{table}

From Table~\ref{tab:Relations}, we can make several conclusions on the behavior
of the subrelation algorithm:
\begin{itemize}
  \item The pair $(X_1,X_2)$ appears to stochastically preserve $\RMonoSum$ for
      all parameter combinations (same parameters in both systems). This
      fact is in fact not hard to verify analytically.
  \item When $\mu_1=\mu_2$, the pair $(X_1,X_2)$ stochastically preserves a nontrivial subrelation of
      $\RSum$. The maximal subrelation of $\RSum$ preserved may depend on
      the model parameters.
  \item $\RSum \cap \RMonoCoord$ appears to be a subrelation of
      $\RSum$ that is stochastically preserved for all parameter choices, as long as
      $\mu_1=\mu_2$ (Figure~\ref{fig:LossSystemTwoLambdas}). This fact is
      proved in~\cite{jonckheere2008}.
  \item  When the system is symmetric ($\lambda_1=\lambda_2$, $\mu_1=\mu_2$
      and $m_1=m_2$), a larger relation $\RSum \cap \RMonoSum \cap
      \RMonoMin \supset \RSum \cap \RMonoCoord$ is stochastically preserved
      (Figure~\ref{fig:LossSystemOneLambda}).
  \item When $\mu_1 \neq \mu_2$, $\RSum$ in general does not have a
      nontrivial subrelation preserved by the pair. This fact is also
      reflected in~\cite[Example 5.2.1]{jonckheere2008}, where it was found
      that the stationary distributions are not in general stochastically ordered
      with respect to $\RSum$.
\end{itemize}
As an illustration, the limiting relations $\RSum \cap \RMonoSum \cap \RMonoMin$ and $\RSum \cap
\RMonoCoord$ (filled circles), together with the initial relation $\RSum$ (filled + unfilled
circles) are plotter in Figures~\ref{fig:LossSystemOneLambda} and~\ref{fig:LossSystemTwoLambdas},
respectively.

\begin{figure}[h]
 \centering
 \includegraphics[width=.6\textwidth]{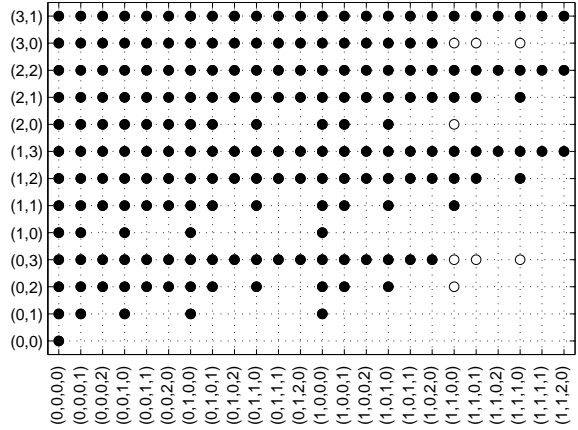}
 \caption{\label{fig:LossSystemOneLambda} Relations $\RSum$ and $M^*(\RSum)$ for a system with
 $m_1=1,m_2=1,n=2$, for $\lambda = (1,1)$ and $\mu = (1,1)$.}
\end{figure}

\begin{figure}[h]
 \centering
 \includegraphics[width=.6\textwidth]{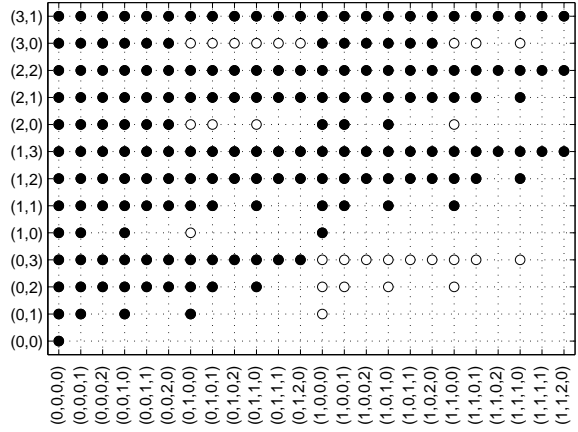}
 \caption{\label{fig:LossSystemTwoLambdas} Relations $\RSum$ and $M^*(\RSum)$ for a system with
 $m_1=1,m_2=1,n=2$, for various $\lambda$ and fixed $\mu=(1,1)$.}
\end{figure}

\subsection{Parallel queues}
\label{sec:ParallelQueues}

Consider a system of two queues in parallel, where customers of queue $k$ have arrival rate
$\lambda_k$ and service rate $\mu_k$. Assuming that all interarrival and service times are
independent and exponential, the queue length process $X=(X_1,X_2)$ is a Markov process in
$\Z_+^2$ with transitions
\[
 x \mapsto \left\{
 \begin{aligned}
  x + e_k, &\quad \text{at rate} \ \lambda_k, \\
  x - e_k, &\quad \text{at rate} \ \mu_k 1(x_k > 0).
 \end{aligned}
 \right.
\]
We shall also consider a modification of the system, where load is balanced by routing incoming
traffic to the shortest queue, modeled as a Markov process $\XLB=(\XLB_1,\XLB_2)$ in $\Z_+^2$ with
transitions
\[
 x \mapsto \left\{
 \begin{aligned}
  x + e_1, &\quad \text{at rate} \ (\lambda_1 + \lambda_2) 1(x_1 < x_2) + \lambda_1 1(x_1 = x_2), \\
  x + e_2, &\quad \text{at rate} \ (\lambda_1 + \lambda_2) 1(x_1 > x_2) + \lambda_2 1(x_1 = x_2), \\
  x - e_1, &\quad \text{at rate} \ \mu_1 1(x_1 > 0), \\
  x - e_2, &\quad \text{at rate} \ \mu_2 1(x_2 > 0).
 \end{aligned}
 \right.
\]

Common sense suggests that load balancing decreases the number of customers in the system in some
sense. The validity of this comparison property can be numerically studied using
Algorithm~\ref{alg:SubrelationTruncated}. Denote the rate matrix of $\XLB$ by $Q_1$ and the rate
matrix of $X$ by $Q_2$, and let $S_n = \Z_+^2 \cap [0,n-1]^2$, $n \ge 1$. Then $S_n$ is a
truncation sequence for both $Q_1$ and $Q_2$.

Figure~\ref{fig:LoadBalancingCoordLessThan} illustrates five iterations of the subrelation
algorithm (computed using~\cite{leskelaSW}) truncated to $S_3$ applied to the coordinatewise order
\[
 \RCoord = \left\{(x,y) \in \Z_+^2 \times \Z_+^2: x_1 \le y_1, \ x_2 \le y_2 \right\},
\]
with the parameters $\lambda_1,\lambda_2,\mu_1,\mu_2$ all equal to one. Because $T_3(R^4) =
T_3(R^5) = \emptyset$, we conjecture that there exists no nontrivial subrelation of $\RCoord$
stochastically preserved by $(\XLB,X)$ (see Theorem~\ref{the:NecessaryCondition}).
\begin{figure}[h]
 \centering
 \includegraphics[width=.8\textwidth]{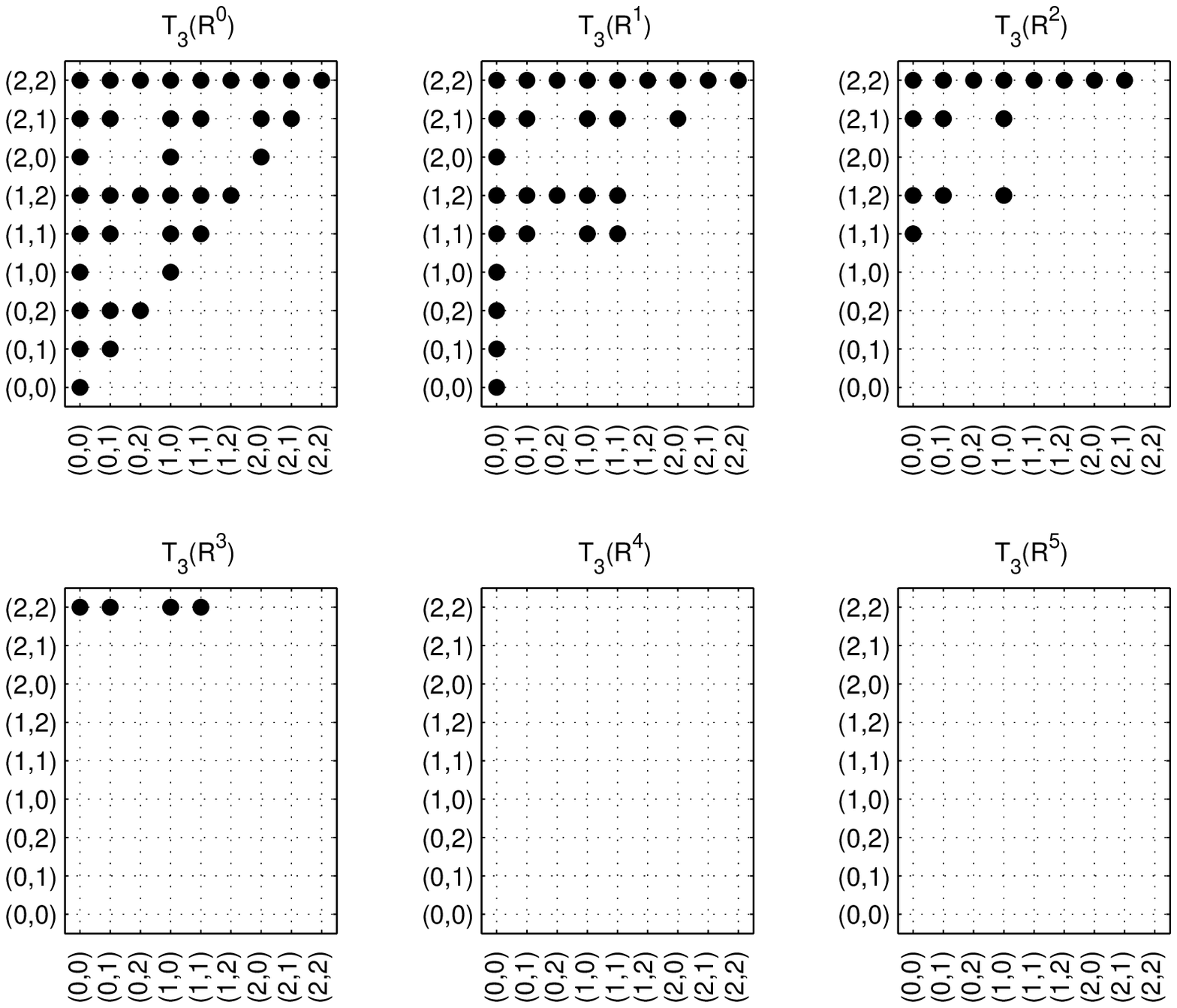}
 \caption{\label{fig:LoadBalancingCoordLessThan} Truncated subrelation algorithm applied to $\RCoord$.}
\end{figure}

Let us next study another order on $\Z_+^2$, defined by
\[
 \RSum = \left\{(x,y) \in \Z_+^2 \times \Z_+^2: x_1 + y_1, \ x_2 + y_2 \right\}.
\]
Figure~\ref{fig:LoadBalancingSumLessThan} illustrates five iterations of the subrelation algorithm
(computed using~\cite{leskelaSW}) truncated to $S_3$ applied to $R^0 = \RSum$ with the parameters
$\lambda_1,\lambda_2,\mu_1,\mu_2$ all equal to one. The observation that $T_3(R^k)$ remains
unchanged from $k=1$ onwards suggests that some nontrivial subrelation of $\RSum$ might be
stochastically preserved by $(\XLB,X)$.
\begin{figure}[h]
 \centering
 \includegraphics[width=.8\textwidth]{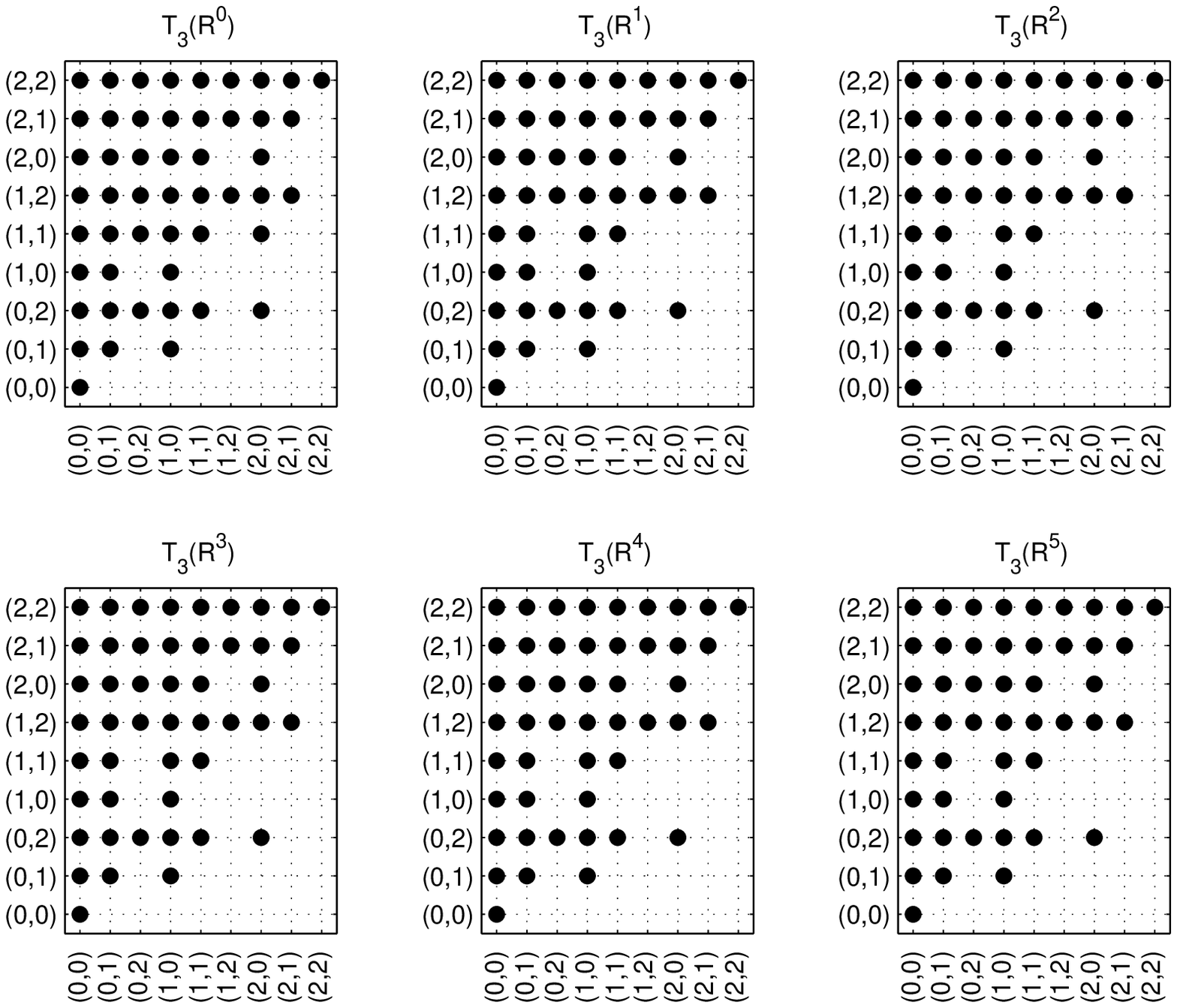}
 \caption{\label{fig:LoadBalancingSumLessThan} Truncated subrelation algorithm applied to $\RSum$.}
\end{figure}
Indeed, it has been analytically shown~\cite{leskelaNow} that whenever $\mu_1=\mu_2$, the
untruncated subrelation algorithm converges to the relation
\[
 R^* = \left\{(x,y): x_1+x_2 \le y_1+y_2 \ \text{and} \ x_1 \vee x_2 \le y_1 \vee y_2 \right\}.
\]
As a consequence of Theorem~\ref{the:maximal}, the pair $(\XLB,X)$ stochastically preserves the
relation $R^*$, which may be identified as the weak majorization order on $\Z_+^2$
\cite{marshall1979}. Especially,
\begin{align*}
 \XLB_1(t) + \XLB_2(t) \, &\lest \, X_1(t) + X_2(t), \\
 \XLB_1(t) \vee \XLB_2(t) \, &\lest \, X_1(t) \vee X_2(t),
\end{align*}
for all $t$, whenever the initial states $\XLB(0)$ and $X(0)$ satisfy the same inequalities.

\section{Conclusions}
\label{sec:Conclusions}

This paper presented computational methods for verifying stochastic relations and finding
relation-invariant couplings of continuous-time Markov processes on finite and countably infinite
state spaces. A key point of the paper is that the stochastic relationship between two probability
measures can be quickly numerically checked, if one of the measures has small support
(Theorem~\ref{the:StochasticRelationFinite}). This result allows the development of a truncation
approach for finding relations stochastically preserved by pairs of Markov processes with locally
bounded jumps. The truncated subrelation algorithm (Algorithm~\ref{alg:SubrelationTruncated})
allows to numerically find candidates for a subrelation of a given relation that is stochastically
preserved by a pair of Markov processes. It remains an interesting open problem for future
research to study how the truncated subrelation algorithm behaves for structured Markov processes
with for example shift-invariant transition rate matrices.

\section{Acknowledgments}
This work has been supported by the Academy of Finland.

\bibliographystyle{abbrv}
\bibliography{lslReferences}

\end{document}